\documentclass{article}

\usepackage{amsfonts}
\usepackage{amsmath}
\usepackage{amssymb}
\usepackage{wasysym}
\usepackage{yfonts}  
\usepackage{amsthm}
\usepackage{amscd}   
\usepackage[all]{xypic}
\usepackage{mathrsfs}
  
\CompileMatrices

\newtheorem{theorem}{Theorem}[section]
\newtheorem{lemma}[theorem]{Lemma}
\newtheorem{proposition}[theorem]{Proposition}
\newtheorem{corollary}[theorem]{Corollary} 
\newtheorem*{definition}{Definition}
\newtheorem*{example}{Example}

\newtheorem*{acknowledgement}{Acknowledgement}
\newtheorem*{thmA}{Theorem 3.5}
\newtheorem*{thmB}{Theorem 3.10}

\numberwithin{equation}{section}

\begin{document}
\title{Groups with the same cohomology as their pro-$p$ completions}
\author{
{\sc Karl Lorensen}\\
\\
Mathematics Department\\
Pennsylvania State University, Altoona College\\
3000 Ivyside Park\\
Altoona, PA 16601-3760\\
USA\\
e-mail: {\tt kql3@psu.edu}
}

\maketitle

\begin{abstract} For any prime $p$ and group $G$, denote the pro-$p$ completion of $G$ by $\hat{G}^p$.  Let $\mathcal{C}$ be the class of all groups $G$ such that, for each natural number $n$ and prime number $p$, $H^n(\hat{G^p},\mathbb Z/p)\cong H^n(G, \mathbb Z/p)$, where $\mathbb Z/p$ is viewed as a discrete, trivial $\hat{G}^p$-module. In this article we identify certain kinds of groups that lie in $\mathcal{C}$. In particular, we show that right-angled Artin groups are in $\mathcal{C}$ and that this class also contains some special types of free products with amalgamation.

\vspace{12pt}

\end{abstract}

\setcounter{section}{-1}

\section{Introduction}

\indent

If $G$ is a group and $p$ a prime number, then $\hat{G}^p$ will denote the pro-$p$ completion of $G$ and $c_G^p:G\to \hat{G}^p$ the completion map.  Let $\mathcal{C}$ be the class of groups $G$ 
such that, for every prime $p$ and nonnegative integer $n$, the homomorphism induced by $c_G^p$ from the continuous cohomology group $H^n(\hat{G}^p,\mathbb Z/p)$ to the discrete cohomology group $H^n(G,\mathbb Z/p)$ is an isomorphism, where $\mathbb Z/p$ is viewed as a discrete, trivial $\hat{G}^p$-module. This class of groups has recently piqued the interest of  researchers in connection with the conjecture, originally due to M. Atiyah, that the $L^2$-Betti numbers of a finite CW-complex whose fundamental group is torsion-free are always integers; see \cite{schick}.

For any group $G$, $H^n(\hat{G}^p,\mathbb Z/p)\cong H^n(G,\mathbb Z/p)$ if $n=0,1$. However, groups for which these cohomology groups are isomorphic in higher dimensions appear to be quite rare. The most obvious examples of groups in the class $\mathcal{C}$ are free groups, since, for $G$ free, both $H^n(\hat{G}^p,\mathbb Z/p)$ and $H^n(G,\mathbb Z/p)$ are trivial for $n\geq 2$. Finitely generated nilpotent groups are also easily seen to be in this class, as we demonstrate below in Corollary 1.3. Other, more exotic, examples of such groups may be found scattered throughout the literature. For instance, in \cite{labute} it is shown that a particular species of one-relator group belongs to $\mathcal{C}$. In addition, P. Linnell and T. Schick \cite{schick} prove, using results from \cite{falk}, that certain kinds of Artin groups reside in this class. Moreover, in \cite{blomer}, Linnell and Schick, together with I. Blomer, establish that primitive link groups lie in $\mathcal{C}$. 
Additional observations concerning groups in $\mathcal{C}$ are contained in \cite{kochloukova} and \cite{weigel}, where three dimensional orientable Poincar\'e duality groups in this class are discussed.  

The goal of the present article is to identify some new types of groups in the class $\mathcal{C}$. First, in Section 2, we prove that right-angled Artin groups are in $\mathcal{C}$; these are groups with a finite presentation in which the relators are all commutators of weight 2 in the generators. 
In order to prove that such groups are in $\mathcal{C}$, we view them as being formed by a finite sequence of HNN extensions and employ the Mayer-Vietoris sequence for an HNN extension, in both its discrete and pro-$p$ incarnations. At the end of Section 2 we apply a similar approach to show that right-angled Artin groups also have another property that is often displayed by groups in $\mathcal{C}$: they are residually $p$-finite for each prime $p$.

In Section 3 we invoke the Mayer-Vietoris sequence for a free product with amalgamation to identify certain free products with amalgamated subgroup that belong to $\mathcal{C}$. Our principal results in that section are the following two theorems.

\begin{thmA} Assume $G_1$ and $G_2$ are groups with a shared finitely generated central subgroup $A$.  If  $G_1/A$ and $G_2/A$ both belong to $\mathcal{C}$, then $G_1\ast_A G_2$ is in $\mathcal{C}$.
\end{thmA}

\begin{thmB} Let $G_1$ and $G_2$ be groups that lie in $\mathcal{C}$ and are residually $p$-finite for every prime $p$.  Then, if $A$ is a cyclic subgroup common to both $G_1$ and $G_2$, $G_1\ast_A G_2$ is in $\mathcal{C}$. 
\end{thmB}

\noindent Examples of groups that satisfy both of the conditions in Theorem 3.10, i.e., that are in $\mathcal{C}$ and residually $p$-finite for every prime $p$, are free groups, finitely generated torsion-free nilpotent groups and-- in view of our results from Section 2-- right-angled Artin groups. 

We conclude Section 3 by describing an example of a free product of two finitely generated torsion-free nilpotent groups with noncyclic, noncentral amalgam that lies outside of the class $\mathcal{C}$. This example demonstrates that our stringent hypotheses in Theorems 3.5 and 3.10 regarding the amalgam cannot be weakened in any significant way.

\section{Notation and preliminary remarks}

In this section we establish the notation and terminology that we will use in the rest of the article. In addition,  we discuss some elementary properties of the class $\mathcal{C}$.

For any two elements $x$ and $y$ in a group $G$, we define  $[x,y]=x^{-1}y^{-1}xy$. Throughout the paper, $p$ will denote a prime natural number. A group with order a finite power of $p$ will be referred to as {\it $p$-finite}. Moreover, a group $G$ is {\it residually $p$-finite} if, for every $g\in G-\{1\}$, there exists a $p$-finite quotient of $G$ in which the image of $g$ is nontrivial. 

If $G$ is a group, $N\unlhd G$, and $[G:N]$ is a finite power of $p$, we write $N\unlhd_{o(p)} G$, where the notation is suggested by the fact that $N$ is open in the pro-$p$ topology on $G$. If $H$ is a subgroup of a group $G$, we will write $H\leq_{c(p)} G$ if $H$ is closed in the pro-$p$ topology on $G$, and we write $H\unlhd_{c(p)} G$ if, in addition, $H$ is normal.

Since we will be working with both pro-$p$ and discrete cohomology, it will be helpful to distinguish between the two types in our notation. Henceforth we will employ $H^*(\  \ ,\  \ )$ for discrete cohomology and $H^*_{\rm cont}(\  \ ,\  \ )$ for pro-$p$ cohomology. Also, $\mathbb Z/p$ will always be viewed as a trivial module with respect to any discrete group and a trivial, discrete topological module with respect to any pro-$p$ group.

If $H$ is a subgroup of the group $G$, we say that $H$ is {\it topologically $p$-embedded} in $G$ and write $H\leq_{t(p)}G$ if the subspace topology on $H$ inherited from the pro-$p$ topology on $G$ coincides with the full pro-$p$ topology on $H$. Note that the following three assertions concerning a subgroup $H$ of a group $G$ are equivalent:

(i) $H\leq_{t(p)} G$;

(ii) for each $N\unlhd_{o(p)} H$, there exists $M\unlhd_{o(p)} G$ such that $M\cap H\leq N$;

(iii) $\hat{H}^p$ embeds in $\hat{G}^p$.

One case of a topological $p$-embedding that is particularly relevant to the subject of this paper pertains to a central group extension whose quotient is in the class $\mathcal{C}$. This situation is treated in the following proposition, which is also proved in \cite{schick}, albeit in a different manner.

\begin{proposition} Assume $A\stackrel{\iota}{\rightarrowtail} G\stackrel{\epsilon}{\twoheadrightarrow} Q$ is a central group extension in which $A$ is finitely generated and $Q$ is in $\mathcal{C}$. Then $\hat{A}^p\stackrel{\hat{\iota}^p}{\rightarrowtail} \hat{G}^p\stackrel{\hat{\epsilon}^p}{\twoheadrightarrow} \hat{Q}^p$ is a short exact sequence of pro-$p$ groups; in other words, $\iota(A)\leq_{t(p)} G$. 
\end{proposition}

\begin{proof} For the sake of simplicity, we assume that $A\leq G$ and $\iota$ is just the inclusion map. First we consider the case where $A$ is a finite $p$-group.  We will prove that $A\leq_{t(p)} G$ by establishing the existence of a normal subgroup $N$ in $G$ such that $[G:N]$ is a power of $p$ and $N\cap A=1$. Let $\xi\in H^2(Q,A)$ be the cohomology class of the extension $A\rightarrowtail G\twoheadrightarrow Q$.  The group $H_{\rm cont}^2(\hat{Q}^p,A)$ is the direct limit of the discrete cohomology with coefficients in $A$ of all the quotients of $\hat{Q}^p$ over open normal subgroups. This means that, since the map $H^2_{\rm cont}(\hat{Q}^p,A)\to H^2(Q,A)$ is surjective, there exists $R\unlhd_{o(p)} Q$ such that the image of $\xi$ in $H^2(R,A)$ is trivial. Thus we have $U\unlhd_{o(p)} G$ containing $A$ such that $U$ splits over $A$. Let $V\unlhd  U$ such that $U=AV$ and $A\cap V=1$. Take $N$ to be the intersection of all the conjugates of $V$ in $G$. Then $N\unlhd G$ and $N\cap A=1$. Moreover, each of these conjugates is normal in $U$ with index a power of $p$, and there are only finitely many of them. Hence $[U:N]$ is a power of $p$, implying that $[G :N]$ is also a power of $p$. It follows, then, that $A\leq_{t(p)} G$.

Now we consider the case where $A$ is an arbitrary finitely generated abelian group. Let $B\unlhd_{o(p)} A$. Since $A$ is finitely generated, $B$ contains a normal subgroup $C$ of $G$ such that $[A:C]$ is a power of $p$. By the first case, $A/C\unlhd_{t(p)} G/C$.
Hence there exists $N\unlhd_{o(p)} G$ such that $A\cap N\leq C\leq B$. Therefore,  $A\leq_{t(p)} G$.
\end{proof}

Proposition 1.1 has the following two corollaries regarding the class $\mathcal{C}$.

\begin{corollary} Let $A\rightarrowtail G\twoheadrightarrow Q$ be a central group extension with $A$ finitely generated. If $A$ and $Q$ are in $\mathcal{C}$, then $G$ is in $\mathcal{C}$.
\end{corollary}

\begin{proof} By Proposition 1.1, $\hat{A}^p\rightarrowtail \hat{G}^p\twoheadrightarrow
 \hat{Q}^p$ is a short exact sequence of pro-$p$ groups for each prime $p$. It follows, then, from the Lyndon-Hochschild-Serre spectral sequences for $A\rightarrowtail G\twoheadrightarrow Q$ and 
$\hat{A}^p\rightarrowtail \hat{G}^p\twoheadrightarrow
 \hat{Q}^p$ that the map $H^n_{\rm cont}(\hat{G}^p,\mathbb Z/p)\to H^n(G,\mathbb Z/p)$ is an isomorphism for each prime $p$ and $n\geq 0$.
 \end{proof}

\begin{corollary} The class $\mathcal{C}$ contains all finitely generated nilpotent groups.
\end{corollary}

\begin{proof} Clearly, any cyclic group is in $\mathcal{C}$. It follows, then, from Corollary 1.2 that every finitely generated abelian group is in the class as well.  Hence, by inducting on the nilpotency class and again applying Corollary 1.2, we can deduce that every finitely generated nilpotent group is in $\mathcal{C}$.
\end{proof}

Even though all finitely generated nilpotent groups are in $\mathcal{C}$,  not every polycyclic group lies in the class. For instance, if $G$ is the semidirect product of $\mathbb Z\oplus \mathbb Z$ with $\mathbb Z$ where $1$ induces the automorphism of $\mathbb Z\oplus \mathbb Z$ that interchanges the components, then, whenever $p$ is odd, $\hat{G}^p\cong \hat{\mathbb Z}^p\oplus \hat{\mathbb Z}^p$, implying that $H_{\rm cont}^2(\hat{G}^p,\mathbb Z/p)\ncong H^2(G,\mathbb Z/p)$.

\section{Right-angled Artin groups}

Right-angled Artin groups are defined as follows.

\begin{definition}{\rm A {\it right-angled Artin group} is any group with a finite generating set $X$ and a presentation of the form
$$\langle X\ |\ [x,y]=1 \ \mbox{\rm for all $(x, y)\in \Sigma$}\ \rangle$$
for some subset $\Sigma$ of the Cartesian product $X\times X$.}
\end{definition}

Our objective in this section is to prove that right-angled Artin groups are in the class $\mathcal{C}$. 
The approach we use in studying these types of groups owes a great deal to \cite{burillo}. It was also employed in \cite{kl} to prove that these groups have the same cohomology as their profinite completions. The key to our method is to view these groups as being formed by a finite sequence of HNN extensions. In analyzing HNN extensions we will employ the following notation: given a discrete group $G$ and an isomorphism $\phi:H\to K$, where both $H$ and $K$ are subgroups of $G$,  the HNN extension of $G$ with respect to $\phi$ is denoted by $G_{\phi}$. 
In other words,
$$G_{\phi}=\langle G, t\ |\ t^{-1}ht=\phi(h)\ \mbox{for all $h\in H$} \rangle .$$

In addition to HNN extensions of discrete groups, we will refer to pro-$p$ HNN extensions.
As described in \cite{profinite}, from any pro-$p$ group $\Gamma$ and any continous isomorphism $\theta: \Delta\to \Lambda$, where $\Delta$ and $\Lambda$ are both closed subgroups of $\Gamma$, we can form the pro-$p$ HNN extension of $\Gamma$ with respect to $\theta$. If $\Gamma$, $\Delta$ and $\Lambda$ are each embedded in the pro-$p$ HNN extension, we refer to the latter as a {\it proper}  
 pro-$p$ HNN extension.  In \cite{profinite} it is observed (without proof) that any proper pro-$p$ HNN extension gives rise to a Mayer-Vietoris sequence that relates the cohomology of the extension to that of the group $\Gamma$ and the subgroup $\Delta$. This sequence is a special case of the Mayer-Vietoris sequence for the fundamental group of a finite graph of pro-$p$ groups, which can be derived from results in \cite{protrees}. 

Our interest is in the case when $\Gamma=\hat{G}^p$, $\Delta=\hat{H}^p$, $\Lambda=\hat{K}^p$ and $\theta=\hat{\phi}^p$, where $G$ is a discrete group with topologically $p$-embedded subgroups $H$, $K$ and $\phi:H\to K$ is an isomorphism. In this case,  $\hat{G}^p_{\phi}$ is the pro-$p$ HNN extension of $\hat{G}^p$ with respect to $\hat{\phi}^p$. Moreover, if $G$ is topologically $p$-embedded in $G_{\phi}$, then the pro-$p$ HNN extension is proper and, therefore, gives rise to a Mayer-Vietoris sequence. This sequence and its relationship to the discrete Mayer-Vietoris sequence for $G_{\phi}$ are described in the following theorem.

\begin{theorem} Let $G$ be a group with isomorphic, topologically $p$-embedded subgroups $H$ and $K$. Assume $\phi:H\to K$ is an isomorphism and $G\leq_{t(p)} G_{\phi}$. Then, for each positive integer $n$,  we have a commutative diagram
\begin{equation} \minCDarrowwidth10pt  \begin{CD}
H_{\rm cont}^{n-1}(\hat{G}^p,\mathbb Z/p)  @>>> H_{\rm cont}^{n-1}(\hat{H}^p,\mathbb Z/p) @>>> H_{\rm cont}^n(\hat{G}^p_{\phi},\mathbb Z/p) @>>> H_{\rm cont}^n(\hat{G}^p,\mathbb Z/p)  @>>> H_{\rm cont}^n(\hat{H}^p,\mathbb Z/p)\\
@VVV @VVV @VVV @VVV @VVV \\
H^{n-1}(G,\mathbb Z/p)  @>>> H^{n-1}(H,\mathbb Z/p) @>>> H^n(G_{\phi},\mathbb Z/p) @>>> H^n(G,\mathbb Z/p) @>>> H^n(H,\mathbb Z/p),
\end{CD} \end{equation}
in which the rows are exact and the vertical maps are induced by the pro-$p$ completion maps for $G$, $H$ and $G_{\phi}$.
\end{theorem}

Theorem 2.1 provides us with the following conditions under which an HNN extension is in the class $\mathcal{C}$.  

\begin{corollary} Let $\phi:H\to K$ be an isomorphism, where $H$ and $K$ are subgroups of a group $G$ that are topologically $p$-embedded for every prime $p$. Assume, further, that $G\leq_{t(p)}G_{\phi}$ for each $p$. If $G$ is in $\mathcal{C}$ and $H$ is in $\mathcal{C}$, then $G_{\phi}$ belongs to $\mathcal{C}$.
\end{corollary}

\begin{proof} In diagram (2.1) the first, second, fourth and fifth vertical maps are isomorphisms, forcing the third map to be one as well. Hence $G_{\phi}$ is in $\mathcal{C}$.
\end{proof}

Unfortunately, it is usually not the case that the base group is topologically $p$-embedded in an HNN extension. However, this does hold for the following type of HNN extension.

\begin{lemma} Let $G$ be a group and $H\leq_{t(p)} G$.  Define the group $\Gamma$ by
$$\Gamma=\langle G,t\ |\  t^{-1}ht=h\ \mbox{for all $h\in H$} \rangle.$$
Then $G\leq_{t(p)} \Gamma$. 
\end{lemma}

\begin{proof} Assume $N\unlhd_{o(p)} G$. Then there is an epimorphism $\theta:\Gamma \to G/N$ mapping $G$ canonically onto $G/N$ and $t$ to $1$. Moreover, $\mbox{Ker}\ \theta\unlhd_{o(p)} \Gamma$ and $\mbox{Ker}\ \theta\cap G\leq N$. Therefore, $G\leq_{t(p)} \Gamma$. 
\end{proof}

The above lemma  immediately yields the

\begin{lemma} Assume $G$ is a group and $H\leq_{t(p)} G$ for every prime $p$ . 
 Define the group $\Gamma$ by
$$\Gamma=\langle G,t\ |\  t^{-1}ht=h\ \mbox{for all $h\in H$} \rangle.$$
If $G$ and $H$ are in $\mathcal{C}$, then $\Gamma$ also lies in $\mathcal{C}$.
\end{lemma}

We now employ Lemma 2.4 to establish that every right-angled Artin group is in $\mathcal{C}$. Before proceeding with the proof, we require the following lemma.

\begin{lemma} If $G$ is a right-angled Artin group with generating set $X$, then, for every $X'\subseteq X$, $\langle X'\rangle \leq_{t(p)} G.$
\end{lemma}

\begin{proof} The proof is by induction on the cardinality of $X$, the case $|X|=1$ being trivial. Assume $|X|>1$. If $X'=X$, then the conclusion follows at once. Assume $X'\neq X$, and let $x\in X-X'$.  Define $H$ to be the group generated by $X-\{ x \}$ with all of the same relators as $G$ except those involving $x$. Furthermore, let $Y$ be the set of all elements in $X-\{ x\}$ that commute with $x$ in $G$. Then
$$G=\langle H, x\ |\ [x,y]=1\ \mbox{for all $y\in Y$}\rangle.$$
Now let $U\unlhd_{o(p)} \langle X'\rangle$. By the inductive hypothesis, $\langle X'\rangle\leq_{t(p)} H$, which means that there exists $N\unlhd_{o(p)} H$ such that $N\cap \langle X'\rangle\leq U$. Moreover, there is a map $\theta:G\to H/N$ that maps $H$ canonically onto $H/N$ and $x$ to $1$. Hence $\mbox{Ker}\ \theta \unlhd_{o(p)} G$ and $\mbox{Ker}\ \theta \cap H\leq N$. Thus $\mbox{Ker}\ \theta \cap \langle X'\rangle\leq U$. It follows, then, that $\langle X'\rangle \leq_{t(p)} G$.
\end{proof}

\begin{theorem} Every right-angled Artin group is in $\mathcal{C}$.
\end{theorem}

\begin{proof}
The proof is by induction on the number of generators, the case of one generator being trivial.  Let $G$ be a right-angled Artin group with generating set $X$ containing more than one element, and assume that every right-angled Artin group with fewer generators than $G$ lies in $\mathcal{C}$.  If $G$ has no relators, then it is free and thus in $\mathcal{C}$. Suppose $G$ has at least one relator, and assume $[x_0,y_0]$ is one of its relators. Let $[x_0,y_0], [x_1,y_0], \cdots , [x_l,y_0]$ be a list of all the relators that involve $y_0$. 
 Now define $H$ to be the group generated by $X-\{y_0\}$ with all of the same relators as $G$ except those involving $y_0$.   In view of the inductive hypothesis, $H$ must belong to $\mathcal{C}$. Moreover, 
$$G=\langle H, y_0\ |\ [x_0,y_0]=[x_1,y_0]=\cdots =[x_l, y_0]=1\rangle.$$
By Lemma 2.5, the subgroup of $H$ generated by $x_0,\cdots , x_l$ is topologically $p$-embedded in $H$. 
Moreover, by the inductive hypothesis, this subgroup must belong to $\mathcal{C}$. Therefore, appealing to Lemma 2.4 allows us to conclude that $G$ lies in $\mathcal{C}$ as well.
\end{proof}

We conclude this section by showing that our conception of right-angled Artin groups as arising from a finite sequence of HNN extensions can also be employed to establish another important property of these groups, namely, that they are residually $p$-finite for every prime $p$.  This property, however, is not a new discovery: it also follows from G. Duchamp and D. Krob's result in \cite{duchamp} that right-angled Artin groups are residually torsion-free nilpotent. Nevertheless, we include our proof as it shares nothing with Duchamp and Krob's and it complements nicely the proof of Theorem 2.6. In essence, our argument mirrors closely that of [{\bf 2}, Theorem 3.4].

We begin with the following concept from the theory of HNN extensions.

\begin{definition}{\rm  Let $G$ be a group with subgroups $H$ and $K$ such that there is an isomorphism $\phi:H\to K$. The sequence $g_0, t^{\epsilon_1}, g_1, t^{\epsilon_2}, \cdots ,t^{\epsilon_n}, g_n$ of elements of $G_{\phi}$ is a {\it reduced sequence} if the following conditions are satisfied: 

(i) $\epsilon_i=\pm 1$ for all $i=0, \cdots, n$;

 (ii) $g_i\in G$ for all $i=0, \cdots, n$;

(iii) the sequence fails to contain any segments of the form $t^{-1}, g_i, t$ with $g_i\in H$ or 

$t, g_i, t^{-1}$ with $g_i\in K$.} 
\end{definition}

Reduced sequences are important for HNN extensions because of the following theorem from \cite{lyndon}. 

\begin{theorem}{\rm ([{\bf 12}, Theorem 2.1])} Let $G$ be a group with subgroups $H$ and $K$ such that there is an isomorphism $\phi:H\to K$. Then the following two statements hold. 

(i) Every element of $G_{\phi}$ has a representation as a product $g_0t^{\epsilon_1}g_1t^{\epsilon_2}\cdots t^{\epsilon_n}g_n,$ where

$g_0, t^{\epsilon_1}, g_1, t^{\epsilon_2}, \cdots ,t^{\epsilon_n}, g_n$ is a reduced sequence . 

(ii) If $g_0, t^{\epsilon_1}, g_1, t^{\epsilon_2}, \cdots ,t^{\epsilon_n}, g_n$ is a reduced sequence with $n\geq 1$, then $$g_0t^{\epsilon_1}g_1t^{\epsilon_2}\cdots t^{\epsilon_n}g_n\neq 1.$$
\end{theorem} 
   
Now we examine two types of HNN extensions that are residually $p$-finite.

\begin{lemma} Let $G$ be a finite $p$-group and $H$ a subgroup of $G$. Define the group $\Gamma$ by
$$\Gamma=\langle G,t\ |\  t^{-1}ht=h\ \mbox{for all $h\in H$} \rangle.$$
Then $\Gamma$ is residually $p$-finite.
\end{lemma}

\begin{proof} There is a homomorphism $\theta:\Gamma \to G$ such that $\theta(g)=g$ for all $g\in G$ and $\theta(t)=1$. Since $\mbox{Ker}\ \theta\cap G=1$, it follows from [{\bf 6}, Theorem 6] that $\mbox{Ker}\ \theta$ is free. Hence $\Gamma$ is an extension of a free group of finite rank by a $p$-finite group and, therefore, residually $p$-finite.
\end{proof}

\begin{lemma} Assume $G$ is a group and $H\leq_{c(p)} G$ for every prime $p$ . 
 Define the group $\Gamma$ by
$$\Gamma=\langle G,t\ |\  t^{-1}ht=h\ \mbox{for all $h\in H$} \rangle.$$
If $G$ is residually $p$-finite, then $\Gamma$ is also residually $p$-finite.
\end{lemma}

\begin{proof} Let $x\in \Gamma-\{1\}$. We need to find a $p$-finite quotient of $\Gamma$ in which the image of $x$ is nontrivial. Let 
$$x=g_0t^{\epsilon_1}g_1t^{\epsilon_2}\cdots t^{\epsilon_n}g_n,$$
where $g_0, t^{\epsilon_1}, g_1, t^{\epsilon_2}, \cdots ,t^{\epsilon_n}, g_n$ is a reduced sequence. Since $H\leq_{c(p)} G$, $G$ contains a normal subgroup $N$ of $p$-power index satisfying the following two properties:

(1) if $g_0\neq 1$, then $g_0\notin N$; 

(2) for each $i=0,\cdots , n$ with $g_i\notin H$,  $g_i\notin NH$. 

\noindent Now define the group $\Sigma$ by
$$\Sigma = \langle G/N, \bar{t}\ |\ \bar{t}^{-1}(Nh)\bar{t}=Nh\ \mbox{for all $h\in H$}\rangle.$$
The group $\Sigma$ is an HNN extension with base group $G/N$ and associated subgroup $NH/N$. 
Moreover, there exists an epimorphism $\theta:\Gamma \to \Sigma$ that maps $G$ canonically onto $G/N$ and $t$ to $\bar{t}$. Also, 
$\theta(g_0), \bar{t}^{\epsilon_1}, \theta(g_1), \bar{t}^{\epsilon_2}, \cdots , \bar{t}^{\epsilon_n}, \theta(g_n)$ is a reduced sequence. Hence, if $n\geq 1$,  $\theta(x)\neq 1$ by Theorem 2.7(ii), and, if $n=0$, then $\theta(x)=\theta(g_0)\neq 1$.  However, by  Lemma 2.8, $\Sigma$ is residually $p$-finite; hence we can find a $p$-finite quotient of $\Gamma$ in which the image of $\theta(x)$ is nontrivial, thus completing the proof. 

\end{proof}

 Lemma 2.9 is the key tool we will employ to prove that right-angled Artin groups are residually $p$-finite. However, first we need to show that the subgroup generated by any subset of the generating set of a right-angled Artin group is closed with respect to the pro-$p$ topology.
 
  \begin{lemma} If $G$ is a right-angled Artin group with generating set $X$, then, for every nonempty subset $X'$ of $X$, $\langle X'\rangle\leq_{c(p)} G$. 
\end{lemma}

\begin{proof} The proof is by induction on the cardinality of $X$, the case $|X|=1$ being trivial. Assume $|X|>1$. If $X'=X$, then the conclusion follows at once. Assume $X'\neq X$, and let $x\in X-X'$.  Define $H$ to be the group generated by $X-\{ x \}$ with all of the same relators as $G$ except those involving $x$. Furthermore, let $Y$ be the set of all elements in $X-\{ x\}$ that commute with $x$ in $G$. Then
$$G=\langle H, x\ |\ [x,y]=1\ \mbox{for all $y\in Y$}\rangle,$$
making $G$ an HNN extension with base group $H$ and associated subgroup $\langle Y\rangle$. Let $g\in G-\langle X'\rangle$. We need to find a quotient of $G$ that is a finite $p$-group in which the image of $g$ is not contained in the image of $\langle X'\rangle$. 
Let 
$$g=h_0t^{\epsilon_1}h_1t^{\epsilon_2}\cdots t^{\epsilon_n}h_n,$$
where $h_0, t^{\epsilon_1}, h_1, t^{\epsilon_2}, \cdots ,t^{\epsilon_n}, h_n$ is a reduced sequence. By the inductive hypothesis, both $\langle X'\rangle$ and $\langle Y\rangle$ are closed subgroups of $H$ with respect to the pro-$p$ topology on $H$. Therefore, $H$ contains a normal subgroup $N$ of $p$-power index satisfying the following two properties:

(1) if $n=0$, then $h_0\notin N\langle X'\rangle$; 

(2) for each $i=0,\cdots , n$ with $h_i\notin \langle Y\rangle$,  $h_i\notin N\langle Y\rangle$. 

\noindent Now form the following group:
$$\Gamma = \langle H/N, \bar{x}\ |\ [\bar{x},Ny]=1\ \mbox{for all $y\in Y$}\rangle.$$
This group is an HNN extension with base group $H/N$ and associated subgroup $N\langle Y\rangle/N$. Thus, in view of Lemma 2.8, $\Gamma$ is residually $p$-finite. Let  $\theta:G\to \Gamma$ be the homomorphism that maps $H$ canonically onto $H/N$ and $x$ to $\bar{x}$. Then
$\theta(h_0), \bar{x}^{\epsilon_1}, \theta(h_1), \bar{x}^{\epsilon_2}, \cdots , \bar{x}^{\epsilon_n}, \theta(h_n)$ is a reduced sequence. We claim that $\theta(g)\notin \theta(\langle X'\rangle )$; it will then follow immediately from the residual $p$-finiteness of $\Gamma$ that there is a $p$-finite quotient of $G$ in which the image of $g$ is not contained in the image of $\langle X'\rangle$. 
If $n=0$, property (1) above implies our claim immediately.  Furthermore, for the case $n>0$ the claim follows from Theorem 2.7(ii).  

\end{proof}

Now we are prepared to prove that right-angled Artin groups are residually $p$-finite.

\begin{theorem} A right-angled Artin group is residually $p$-finite for every prime $p$.
\end{theorem}

\begin{proof}
The proof is by induction on the number of generators, the case of one generator being trivial.  Let $G$ be a right-angled Artin group with generating set $X$ containing more than one element, and assume that every right-angled Artin group with fewer generators than $G$ is residually $p$-finite.  If $G$ has no relators, then it is free and thus residually $p$-finite. Suppose $G$ has at least one relator, and take $[x_0,y_0]$ to be one of its relators. Let $[x_0,y_0], [x_1,y_0], \cdots , [x_l,y_0]$ be a list of all the relators that involve $y_0$. 
 Now define $H$ to be the group generated by $X-\{y_0\}$ with all of the same relators as $G$ except those involving $y_0$.   Because of the inductive hypothesis, $H$ must be residually $p$-finite. Moreover, 
$$G=\langle H, y_0\ |\ [x_0,y_0]=[x_1,y_0]=\cdots =[x_l, y_0]=1\rangle.$$
By Lemma 2.10, the subgroup of $H$ generated by $x_0,\cdots , x_l$ is closed in $H$ with respect to the pro-$p$ topology. Therefore, invoking Lemma 2.9, we can conclude that $G$ is residually $p$-finite .
\end{proof}

\section{Free products with amalgamation}

In this section we prove that certain free products with amalgamated subgroup are in $\mathcal{C}$. Our arguments are based on the Mayer-Vietoris sequence  for free products with amalgam, in both its discrete and pro-$p$ versions. We begin by recalling some facts about pro-$p$ free products with amalgamation from \cite{profinite}. If $\Gamma_1$ and $\Gamma_2$ are pro-$p$ groups with a common closed subgroup $\Delta$, then we can always form the pro-$p$ free product of $\Gamma_1$ and $\Gamma_2$ with amalgamated subgroup $\Delta$; this is the pushout of $\Gamma_1$ and $\Gamma_2$ over $\Delta$ in the category of pro-$p$ groups. If $\Gamma_1$, $\Gamma_2$ and $\Delta$ are all embedded in this pushout, which is by no means always the case, then the latter is referred to as a {\it proper} pro-$p$ free product with amalgamation. Associated to a proper pro-$p$ free product with amalgamation is a  Mayer-Vietoris sequence that relates the cohomologies of the various groups to one another. Like the HNN version, it is really a special case of the Mayer-Vietoris sequence associated to a finite graph of pro-$p$ groups, mentioned in the previous section

 We are interested in the special situation when $\Gamma_1=\hat{G}_1^p$, $\Gamma_2=\hat{G}_2^p$ and $\Delta=\hat{H}^p$, where $G_1$ and $G_2$ are discrete groups with a shared topologically $p$-embedded subgroup $H$. In this case, the pro-$p$ completion of $G=G_1\ast _H G_2$ is the pro-$p$ free product of $\hat{G}_1^p$ and $\hat{G}_2^p$ with amalgamated subgroup $\hat{H}^p$. Moreover, this pro-$p$ free product with amalgam is proper if and only if both $G_1$ and $G_2$ are topologically $p$-embedded in $G$. In this case, we have a Mayer-Vietoris sequence for $\hat{G}^p$ relating the cohomologies of $\hat{G}_1^p$, $\hat{G}_2^p$ and $\hat{H}^p$. This sequence is described in the following theorem, which also illuminates the connection to the discrete Mayer-Vietoris sequence for $G$.

 \begin{theorem} Let $G_1$ and $G_2$ be groups with a common subgroup $H$ that is topologically 
 $p$-embedded in both groups, and let $G=G_1\ast_H G_2$.  Assume, further, that $G_1$ and $G_2$ are both topologically $p$-embedded in $G$. Then, for each positive integer $n$, we have a commutative diagram
 \begin{equation}   \begin{CD}
H^{n-1}(\hat{G}_1^p,A)\oplus H^{n-1}(\hat{G}_2^p,A)   @>>> H^{n-1}(G_1,A)\oplus H^{n-1}(G_2,A) \\
@VVV @VVV \\
H^{n-1}(\hat{H}^p,A) @>>> H^{n-1}(H,A) \\
@VVV @VVV\\
H^n(\hat{G}^p,A) @>>>  H^n(G,A) \\
@VVV @VVV \\
H^n(\hat{G}_1^p,A)\oplus H^n(\hat{G}_2^p,A)   @>>> H^n(G_1,A)\oplus H^n(G_2,A) \\
@VVV @VVV\\
H^n(\hat{H}^p,A) @>>> H^n(H,A), 
\end{CD} \end{equation}
in which the columns are exact and the horizontal maps are induced by the pro-$p$ completion maps for $G_1$, $G_2$, $H$ and $G$.
\end{theorem}

The above theorem allows us to stipulate a set of conditions under which a free product with amalgamation lies in the class $\mathcal{C}$. 

\begin{corollary} Let $G_1$ and $G_2$ be groups with a common subgroup $H$ that is topologically $p$-embedded in both groups for each prime $p$. Furthermore, assume $G_1$ and $G_2$ are both topologically $p$-embedded in $G=G_1\ast_H G_2$ for every $p$. If $G_1$, $G_2$ and $H$ are all in $\mathcal{C}$, then $G$ also belongs to $\mathcal{C}$.
\end{corollary}

\begin{proof}  The first, second, fourth and fifth horizontal maps in (3.1) are isomorphisms, forcing the third to be one as well. 
\end{proof}

The catch with Corollary 3.2 is that it is difficult to show that the factors in a free product with amalgamation are topologically $p$-embedded. However, an important situation in which this holds is described in the following lemma, similar to [{\bf 14}, Theorem 3.1]. 

\begin{lemma} Let $G_1$ and $G_2$ be groups with a common subgroup $H$. Assume that, for each pair $\{N_1, N_2\}$ with $N_i\unlhd_{o(p)} G_i$, there exists a pair $\{P_1, P_2\}$ such that $P_i\unlhd_{o(p)} G_i$, $P_i\leq N_i$ and $P_1\cap H=P_2\cap H$. Then $G_1$ and $G_2$ are topologically $p$-embedded in $G_1\ast_H G_2$. 
\end{lemma}

To prove the above lemma, we require the following theorem of G. Higman \cite{higman}.

\begin{theorem}{\rm (Higman)} Let $G_1$ and $G_2$ be finite $p$-groups with a common cyclic subgroup $A$. Then $G_1\ast_A G_2$ is residually a finite $p$-group.
\end{theorem}

\begin{proof}[Proof of Lemma 3.3.] Let $G=G_1\ast_H G_2$. Assume $N_1\unlhd_{o(p)} G_1$ and $N_2\unlhd_{o(p)}G_2$. Then there exists a pair $\{P_1, P_2\}$ such that $P_i\unlhd_{o(p)} G_i$, $P_i\leq N_i$ and $P_1\cap H=P_2\cap H$. Since  
$P_1\cap H=P_2\cap H$, $P_1H/P_1\cong P_2H/P_2$. We can then identify these two groups via this isomorphism and form the free product with amalgamation
$$\bar{G} = G_1/ P_1\ast_{P_1H/P_1} G_2/P_2.$$
Moreover, there is an epimorphism $\theta: G\to \bar{G}$ that maps $G_1$ and $G_2$ canonically onto $G_1/ P_1$ and $G_2/ P_2$, respectively. Also, by Higman's theorem, there is an epimorphism $\epsilon$ from $\bar{G}$ onto a finite $p$-group such that the restriction of $\epsilon$ to $G_i/P_i$ is injective. Let $K=\mbox{Ker}\ \epsilon \theta$. Then $K\unlhd_{o(p)} G$ and $K\cap G_i\leq P_i\leq N_i$. Therefore, 
$G_i\leq_{t(p)} G$.
\end{proof}

Unfortunately, it is not easy to recognize when the hypotheses of Lemma 3.3 might be satisfied. Nevertheless, we will discern two important cases where these conditions are fulfilled. The first involves a central amalgam and is treated in the following theorem.

\begin{theorem} Assume $G_1$ and $G_2$ are groups with a shared finitely generated central subgroup $A$.  If  $G_1/A$ and $G_2/A$ both belong to $\mathcal{C}$, then $G_1\ast_A G_2$ is in $\mathcal{C}$.
 \end{theorem}
 
 \begin{proof} 
 
By Corollaries 1.2 and 1.3, $G_1$ and $G_2$ are both in $\mathcal{C}$. Our plan is to use Lemma 3.3 to establish that $G_1$ and $G_2$ are topologically $p$-embedded in $G_1\ast_A G_2$; the conclusion of the theorem will then follow by Corollary 3.2. In order to invoke Lemma 3.3, we need to show that, for each pair $\{N_1, N_2\}$ with $N_i\unlhd_{o(p)} G_i$,  there exists a pair $\{P_1,P_2\}$ such that $P_i\unlhd_{o(p)} G_i$, $P_i\leq N_i$ and $P_1\cap A=P_2\cap A$. Assume $N_1\unlhd_{o(p)} G_1$ and $N_2\unlhd_{o(p)} G_2$. Take $U$ to be the intersection of the images of $A\cap N_1\cap N_2$ under all the automorphisms of $A$, keeping in mind that there are only finitely many such images that are distinct. Then $U$ is a normal subgroup of both $G_1$ and $G_2$  contained in $A\cap N_1\cap N_2$, and the index of $U$ in $A$ is a power of $p$. By Proposition 1.1, $A\leq_{t(p)} G_i$. Hence we can find $M_i\unlhd_{o(p)} G_i$ such that $M_i\leq N_i$ and $M_i\cap A\leq U$. Now we let $P_i=UM_i$. Then $P_i\cap A=U(M_i\cap A)=U$; moreover, $P_i\unlhd_{o(p)} G_i$ and $P_i\leq N_i$. Thus we have constructed the desired pair $\{P_1,P_2\}$.
 \end{proof} 

The second situation where the hypotheses of Lemma 3.3 are satisfied is when the amalgamated subgroup is cyclic and the groups are residually $p$-finite. To verify this, we avail ourselves of the following property of residually $p$-finite groups, observed in \cite{kim}.

\begin{lemma} Let $G$ be a group that is residually $p$-finite, and let $a\in G$. Then, for every $n\in \mathbb N$, there exists $N_n\unlhd_{o(p)} G$ such that $N_n\cap \langle a\rangle = \langle a^{p^n}\rangle $.  
\end{lemma}

The proof of the above lemma is based on the following elementary property of finite $p$-groups.

\begin{lemma} Let $G$ be a finite $p$-group and $a\in G$ with $|a|=p^n$. Then, for every $i=0, \cdots, n$, there exists a quotient of $G$ in which the image of $a$ has order $p^i$.
\end{lemma}

\begin{proof} Let $|G| =p^m$ and proceed by induction on $m$. The cases $m=0,1$ are vacuously true. Assume $m>1$. Then $G$ contains a central subgroup $N$ of order $p$. In the factor group $G/N$ the order of $Na$ is either $p^{n-1}$ or $p^n$. Hence, by the inductive hypothesis, for every $i=0, \cdots , n-1$, $G/N$ has a quotient in which the image of $Na$ has order $p^i$. Therefore, the conclusion holds.
\end{proof}

Equipped with the above lemma, we are ready to prove Lemma 3.6.

\begin{proof}[Proof of Lemma 3.6.]  First we consider the case where $a$ has infinite order. Assume $n\in \mathbb N$. The residual property of $G$ allows us to find $M_n\unlhd_{o(p)} G$ such that the order of the image of $a$ in $G/M_n$ is greater than $p^n$. Applying the above lemma to the factor group $G/M_n$, we can find $N_n \unlhd_{o(p)} G$ such that the order of the image of $a$ in $G/N_n$ is exactly $p^n$. Thus the subgroup $N_n$ has the desired property.

Now assume $a$ has finite order, i.e., $|a|=p^k$ for some $k\geq 0$. Then, appealing to the residual property of $G$, we can find $N\unlhd_{o(p)} G$ such that the order of $Na$ in $G/N$ is exactly $p^k$. Now for $n\geq k$ we let $N_n=N$.  Moreover, for $0\leq n<k$ we apply Lemma 3.7 to $G/N$, obtaining $N_n\unlhd_{o(p)} G$ such that the order of the image of $a$ in $G/N_n$ is exactly~$p^n$. The subgroups $N_n$, then, enjoy the properties we seek. 
\end{proof}

The following property of residually $p$-finite groups is an immediate consequence of Lemma 3.6.

\begin{lemma} If $G$ is a residually $p$-finite group, then every cyclic subgroup of $G$ is topologically $p$-embedded in $G$. 
\end{lemma}

In addition, Lemma 3.6 allows us to prove that, in free products with cyclic amalgamation, if both factors are residually $p$-finite, then they are each topologically $p$-embedded.  

\begin{lemma} Assume $G_1$ and $G_2$ are residually $p$-finite groups with a common cyclic subgroup $A$. Then $G_1$ and $G_2$ are both topologically $p$-embedded in $G_1\ast_A G_2$.
\end{lemma}

\begin{proof} We need to show that, for each pair $\{N_1,N_2\}$ with $N_i\unlhd_{o(p)} G_i$, there exists a pair $\{P_1,P_2\}$ such that $P_i\unlhd_{o(p)} G_i$, $P_i\leq N_i$ and $P_1\cap A=P_2\cap A$. Suppose  $N_1\unlhd_{o(p)} G_1$ and $N_2\unlhd_{o(p)} G_2$. By Lemma 3.6,  for each $i=1, 2$, there exists $M_i\unlhd_{o(p)} G_i$ such that $$M_i\cap A = A\cap N_1\cap N_2. $$ Now, if we take $P_i=M_i\cap N_i$, then the pair $\{P_1,P_2\}$ has the desired properties.
\end{proof}

In conjunction with Corollary 3.2, the above lemma yields immediately the following theorem.

\begin{theorem} Assume $G_1$ and $G_2$ are groups in $\mathcal{C}$ with a common cyclic subgroup $A$.  Suppose, further, that $G_1$ and $G_2$ are residually $p$-finite for each prime $p$. Then $G_1\ast_A G_2$ is also in $\mathcal{C}$.
\end{theorem}

To illuminate the significance of Theorem 3.10, we list several examples of groups that satisfy the hypotheses, i.e., that are residually $p$-finite for every prime $p$ and in $\mathcal{C}$:
\vspace{5pt}

1. Free groups.
\vspace{5pt}

2. Finitely generated torsion-free nilpotent groups.
\vspace{5pt}

3. Right-angled Artin groups.        
\vspace{5pt}

4. A free product of groups of type 1 or 2 with maximal cyclic amalgam. (That such a product is residually $p$-finite for every prime $p$ is established in \cite{kim}.)
\vspace{5pt}

5. Any group $G$ with a central subgroup $A$ such that $A$ is free abelian of finite rank and $G/A$ is both in $\mathcal{C}$ and residually $p$-finite for every prime $p$.
\vspace{5pt}

The last observation follows from Proposition 1.1. To see this, let $Q=G/A$ and consider the short exact sequence $\hat{A}^p\rightarrowtail \hat{G}^p\twoheadrightarrow \hat{Q}^p$ guaranteed by the proposition. Since the completion maps $c^p_A$ and $c^p_Q$ are injections, it follows that $c_G^p$ is also injective, implying that $G$ is residually $p$-finite. Moreover, $G$ is in $\mathcal{C}$ by virtue of Corollaries 1.2 and 1.3.  

In conclusion, we present the following example illustrating that, in both Theorems 3.5 and 3.10, we cannot dispense with the conditions placed on the amalgam. The pro-$p$ completions of the groups involved in our example are the groups adduced by L. Ribes [{\bf 14}, Section 4] (also  [{\bf 13}, Example 9.2.9]) as yielding an instance of a pro-$p$ free product with amalgam that is not proper. 

\begin{example}{\rm Set $N=\mathbb Z\oplus \mathbb Z$. Take $G_1$ to be the semidirect product of $N$ with $\mathbb Z$, where $1$ induces the automorphism $(a,b)\mapsto (a+b,b)$ of $N$.  Furthermore, let $G_2$ be the semidirect product of $N$ with $\mathbb Z$, where $1$ induces the automorphism $(a,b)\mapsto (a,a+b)$ of $N$. Then the isomorphic groups $G_1$ and $G_2$ are finitely generated torsion-free nilpotent groups of class 2. 

Let $G=G_1\ast_N G_2$. Our intention is to prove that $H_{\rm cont}^2(\hat{G}^p,\mathbb Z/p)=0$ whereas $H^2(G,\mathbb Z/p)\neq~0$, thereby establishing that $G$ lies outside of the class $\mathcal{C}$. First we observe that there is a split extension $N\rightarrowtail G\twoheadrightarrow F_2$, where $F_2$ denotes the free group on two generators. We now proceed to argue that the image of $N$ in $\hat{G}^p$ is trivial, which will imply that  $\hat{G}^p\cong \hat{F}_2^p$. We begin with an arbitrary epimorphism $\theta:G\to P$ such that $P$ is $p$-finite.  Set $M=\mbox{Ker}\ \theta \cap N$. Then $M$ is invariant under both the automorphisms $(a,b)\mapsto (a+b,b)$ and $(a,b)\mapsto (a,a+b)$ of $N$. It is not difficult to see that this means that, whenever $(a,b)\in M$, both $(a,0)$ and $(0,b)$ must be elements of $M$. Thus $M=p^n\mathbb Z\oplus p^n\mathbb Z$ for some nonnegative integer $n$. It is apparent from this description of $M$ that $N/M$ fails to contain any nontrivial element that is fixed by the action of $F_2$; in other words, $\theta(N)\cap Z(P)=1$. However, since $P$ is $p$-finite, this implies that $\theta(N)=1$. Therefore, $N$ has a trivial image in $\hat{G}^p$, so that $\hat{G}^p\cong \hat{F}_2^p$. Hence $H_{\rm cont}^2(\hat{G}^p,\mathbb Z/p)=0$.

Finally, we establish that $H^2(G,\mathbb Z/p)\neq 0$. 
For this we require the following segment of the Mayer-Vietoris sequence for $G$.
\begin{displaymath}
 H^2(G,\mathbb Z/p)\to H^2(G_1,\mathbb Z/p)\oplus H^2(G_2,\mathbb Z/p)\to H^2(N,\mathbb Z/p)
\end{displaymath}
Since $H^2(N,\mathbb Z/p)\cong \mathbb Z/p$ and $H^2(G_1, \mathbb Z/p)\cong H^2(G_2, \mathbb Z/p)$, it will follow that $H^2(G, \mathbb Z/p)\neq 0$ if we can show that $H^2(G_1,\mathbb Z/p)\neq 0$. 
To investigate $H^2(G_1, \mathbb Z/p)$, we use the Lyndon-Hochschild-Serre spectral sequence for the extension $N\rightarrowtail G_1\twoheadrightarrow \mathbb Z$. In this spectral sequence, $E^{11}_{\infty}=E^{11}_2\cong H^1(\mathbb Z, \mathbb Z/p \oplus \mathbb Z/p)$, where the action of $\mathbb Z$ on $\mathbb Z/p \oplus \mathbb Z/p$ causes $1$ to induce the automorphism $(a,b)\mapsto (a+b,b)$ of $\mathbb Z/p \oplus \mathbb Z/p$.  Employing the interpretation of the first cohomology group as derivations modulo inner derivations, the latter group is readily seen to be isomorphic to $\mathbb Z/p$. Consequently, $H^2(G_1, \mathbb Z/p)\neq 0$, implying that
$H^2(G,\mathbb Z/p)\neq 0$. }

\end{example}

\begin{acknowledgement}{\rm We are grateful to the anonymous referee for his/her helpful comments.}
\end{acknowledgement}

\end{document}